\documentclass[10pt,twoside,reqno]{amsart}
\usepackage[all]{xy}
        \usepackage {amssymb,latexsym,amsthm,amsmath,mathtools,dsfont,mathrsfs}
        \usepackage{enumitem,color}

\usepackage[
  top=1in,
  bottom=1in,
  left=1.2in,
  right=1.2in
]{geometry}
\usepackage{makecell}
\usepackage{mathtools}

\usepackage{longtable}
\usepackage{framed}

\usepackage[backend=bibtex,backref=true,maxbibnames=999,doi=false,isbn=false,url=false
]{biblatex} 
\usepackage{float}
\usepackage{hyperref}
\long\def\symbolfootnote[#1]#2{\begingroup%
\def\thefootnote{\fnsymbol{footnote}}\footnote[#1]{#2}\endgroup}

\makeatletter
\def\imod#1{\allowbreak\mkern10mu({\operator@font mod}\,\,#1)}
\makeatother
\makeatletter
\renewcommand*\env@matrix[1][*\c@MaxMatrixCols c]{%
  \hskip -\arraycolsep
  \let\@ifnextchar\new@ifnextchar
  \array{#1}}
\makeatother
\usepackage[capitalize,nameinlink]{cleveref} 
\usepackage{comment} 
\usepackage{xcolor} 
\hypersetup{
    colorlinks,
    linkcolor={red!80!black},
    citecolor={green!80!black},
    urlcolor={blue!80!black}
}

\newtheorem{theorem}{Theorem}[section]
\newtheorem{lemma}[theorem]{Lemma}

\newtheorem*{theorem*}{Theorem}
\theoremstyle{definition}

\newtheorem{example}[theorem]{Example}

\newtheorem{conjecture}[theorem]{Conjecture}

\numberwithin{equation}{section}
\newcommand{\ignore}[1]{}

\newcommand{\mynote}[1]{}
\title[A note on idempotents in quandle rings]{A note on idempotents in quandle rings}
\author[Dilpreet Kaur]{Dilpreet Kaur}
\email{dilpreetkaur@iitj.ac.in}
\address{Indian Institute of Technology Jodhpur
N.H. 62, Nagaur Road, Karwar Jodhpur 342030
Rajasthan}

\author[Pushpendra Singh]{Pushpendra Singh}
\email{pushpendra@iisermohali.ac.in}
\address{Indian Institute of Science Education and Research Mohali, Sector 81, Mohali 140306, India}
\thanks{}
\date{\today}
\subjclass[2020]{17D99,20N02}
\keywords{Quandles, Quandle rings, Affine quandles}
\begin{document}
\setcounter{section}{0}
\begin{abstract}
It was conjectured that the nonzero idempotents of the integral quandle ring of a finite latin quandle are trivial. We prove this holds for latin dihedral quandles. Furthermore, we discuss counterexamples against this conjecture.
\end{abstract}
\maketitle
\section{Introduction}
A quandle is an algebraic structure $(Q,\triangleright)$ where the binary operation satisfies the following properties.
\begin{enumerate}
\item 
The map $R_x: Q \rightarrow Q,~R_x(y)=y \triangleright x$ is an automorphism for all $x \in Q$.
\item 
$R_x(x)=x$ for all $x\in Q$.
\end{enumerate}

An affine quandle ${\rm Aff}(A,f)$ is an additive abelian group $A$ with $f \in {\rm Aut}(A)$ with quandle operation $x\triangleright y=f(x)+(1-f)(y)$. In particular, if $A=\mathbb Z_n$ and $f(x)=-x$, then the affine quandle is known as a dihedral quandle. We denote it by $\mathcal{R}_n$.

A quandle ring is defined analogously to a group ring. It is introduced in \cite{BPS19}. Let $K$ be an associative ring. Consider the following set of finite linear combinations.
\[K[Q] = \left\{\sum\limits_{i}a_xx\;; a_x \in K,x \in Q \right\}\]
Then $K[Q]$ forms an abelian group with the usual addition. The following multiplication equips it with a ring structure.
\[
\left( \sum\limits_{x}a_x x  \right) \cdot \left( \sum\limits_{y}b_y y  \right)= \sum\limits_{x,y} a_xb_y(x\triangleright y) 
\]

The mapping $\varepsilon: K[Q]\rightarrow K, \sum\limits_{x}a_x x \mapsto \sum\limits_{x}a_x$ is a ring homomorphism. This is known as augmentation mapping.

The authors give various results about quandles $Q$ by studying properties of $K[Q]$. Furthermore, the authors in  \cite{ENS23} and \cite{ENSD23} study the idempotent elements of $K[Q]$ and give applications of idempotents in knot theory for the construction of invariants of knots and links. In particular, they state the following conjectures.

\begin{conjecture}\cite[Conjecture 3.10]{ENSD23} \label{c1}
The integral quandle ring of a semi-latin quandle has only trivial idempotents. In particular, the integral quandle ring of a finite latin quandle has only trivial idempotents.
\end{conjecture}

\begin{conjecture} \cite[Conjecture 7.6]{ENS23} \label{c2}
The Peirce spectrum of the complex quandle algebra of the dihedral quandle of odd order is $\{0,1,-1\}$.
\end{conjecture}

In the section \ref{sec2}, we prove that the integral quandle ring of latin dihedral quandles has only trivial idempotents. Furthermore, in section \ref{sec3}, we compute the Peirce spectrum of $\mathbb C[\mathcal R_n]$ for $n$ odd. In section \ref{sec4}, we show that for an integral ring of affine connected quandles of prime order, we can construct infinitely many nontrivial idempotents given a nontrivial idempotent.

\section{Idempotents in $\mathbb Z[\mathcal R_n]$} \label{sec2}

The dihedral quandle $\mathcal R_n$ is a latin quandle for $n$ odd. In this section, we assume that $n$ is odd. The quandle operation on $\mathbb Z_n$ is $x\triangleright y=2y-x$.

The quandle ring $\mathbb Z[\mathcal R_n]$ is a free $\mathbb Z$-module with basis $\{e_x\;;x\in \mathcal R_n \}$ with the multiplication on basis elements given as $e_xe_y=e_{x\triangleright y}=e_{2y-x}$. The mapping $x\mapsto e_x$ is an embedding of quandle $\mathcal R_n$ into $\mathbb Z[\mathcal R_n]$.

Let $u=\sum\limits_{x\in \mathbb Z_n}a_xe_x$ be an element in $\mathbb Z[\mathcal R_n]$. We solve $u^2=u$ and find the coefficients $a_x$.
\begin{align*}
u^2= \left( \sum\limits_{x\in \mathbb Z_n}a_xe_x \right)\cdot \left( \sum\limits_{y\in \mathbb Z_n} a_ye_y \right)&= \sum\limits_{x,y\in \mathbb Z_n} a_xa_y e_{2y-x} \\
&=\sum\limits_{y,z\in \mathbb Z_n}a_{2y-z}a_ye_z
\end{align*} 
Let $u^2=\sum\limits_{z\in \mathbb Z_n}c_ze_z$, then we have $c_z=\sum\limits_{y\in \mathbb Z_n}a_{2y-z}a_y$. The equality $u^2=u$ means
\[ a_z= \sum\limits_{y\in \mathbb Z_n} a_{2y-z}a_y \quad \quad   \text{ for all } z \in \mathbb Z_n\]

We solve the above system of equations using fourier transform on $\mathbb Z_n$. Let $\zeta=e^{i2\pi/n}$. Using the coefficients of $u$, define a map $\hat{a}: \mathbb Z_n \rightarrow \mathbb C, r \mapsto \sum\limits_{x\in \mathbb Z_n}a_x \zeta^{rx}$.

\begin{lemma} \label{ortho}
Let $x,y\in \mathcal R_n$. Then
\[ \sum\limits_{r\in \mathbb Z_n}\zeta^{r(x-y)}=\begin{cases}
    n \text{ if } x=y\\
    0 \text{ if } x\neq y
\end{cases} \]
\end{lemma}
\begin{proof}
The expression gives $n$ for $x=y$. For $x\neq y$, let $q=\zeta^{x-y}$. Then $q\neq 1$. We have $\sum\limits_{r\in \mathbb Z_n}q^r = \frac{q^n-1}{q-1}=0$, since $q^n=1$. 
\end{proof}

\begin{lemma}\label{inv}
Let $x\in \mathcal R_n$. Then we have
\[a_x= \frac{1}{n} \sum\limits_{r\in \mathbb Z_n} \hat{a}(r)\zeta^{-rx}\]
Furthermore,
\[ \sum\limits_{r\in \mathbb Z_n} |\hat{a}(r)|^2=n \sum\limits_{x\in \mathbb Z_n} a_x^2 \]
\end{lemma}

\begin{proof}
Expanding \textit{rhs} by substituting for $\hat{a}(r)$, we get
\[\frac{1}{n} \sum\limits_{r\in \mathbb Z_n} \hat{a}(r)\zeta^{-rx}= \frac{1}{n}\sum\limits_{r,y\in \mathbb Z_n}a_y\zeta^{ry}\zeta^{-rx}=\frac{1}{n}\sum\limits_{y\in \mathbb Z_n}a_y \sum\limits_{r\in \mathbb Z_n}\zeta^{r(y-x)}=a_x\]
where the last equality follows using Lemma \ref{ortho}.
Now we have 
\begin{align*}
\sum\limits_{r\in \mathbb Z_n} |\hat{a}(r)|^2= \sum\limits_{r\in \mathbb Z_n} \hat{a}(r)\overline{\hat{a}(r)}&= \sum\limits_{r\in \mathbb Z_n} \left( \sum\limits_{x\in \mathbb Z_n}a_x\zeta^{rx}  \right) \left( \sum\limits_{y\in \mathbb Z_n} \bar{a_y} \zeta^{-ry} \right)\\
&= \sum\limits_{x,y\in \mathbb Z_n}a_xa_y \sum\limits_{r\in \mathbb Z_n}\zeta^{r(x-y)}= n\sum\limits_{x\in \mathbb Z_n} a_x^2\\
\end{align*}
\end{proof}

Now using coefficients of $u^2$, we define function $\hat{c}: \mathbb Z_n \rightarrow \mathbb C, r \mapsto \sum\limits_{z\in \mathbb Z_n}c_z\zeta^{rz}$. 

\begin{lemma}\label{idemeq}
Let $u^2=u$. Then $\hat{a}(r)=\hat{a}(-r)\hat{a}(2r)$.
\end{lemma}

\begin{proof}
Substituting for $c_z$ in $\hat{c}(r)$, we get
\[ \hat{c}(r)=\sum\limits_{z\in \mathbb Z_n} \left( \sum\limits_{y\in \mathbb Z_n}a_{2y-z}a_y \right) \zeta^{rz}\]
Put $x=2y-z$, then we have
\begin{align*}
\hat{c}(r)&=\sum\limits_{x,y\in \mathbb Z_n}a_xa_y \zeta^{r(2y-x)}\\
&= \left( \sum\limits_{x\in \mathbb Z_n}a_x\zeta^{-rx}  \right) \cdot \left( \sum\limits_{y\in \mathbb Z_n}a_y \zeta^{2ry} \right)\\
&= \hat{a}(-r)\hat{a}(2r)
\end{align*}
Now $u^2=u$ implies $c_z=a_z$ for all $z\in \mathbb Z_n$, and so $\hat{c}(r)=\hat{a}(r)$. Thus $\hat{a}(r)=\hat{a}(-r)\hat{a}(2r)$ for all $r\in \mathbb Z_n$.
\end{proof}

Now, below we work under the condition that $u^2=u$.
\begin{lemma} \label{abval}
Let $n$ be odd, then for $r \in \mathbb Z_n$, either $\hat{a}(r)=0$ or $|\hat{a}(r)|=1$.
\end{lemma}

\begin{proof}
We have $\hat{a}(-r)=\overline{\hat{a}(r)}$. Let $\overline{\hat{a}(r)} \neq 0$. Then, by the previous lemma, we have $\hat{a}(2r)=\frac{\hat{a}(r)}{\overline{\hat{a}(r)}}$ and so $|\hat{a}(2r)|=1$. Repeating this, we get $|\hat{a}(2^jr)|=1$ for all $j\geq 1$. Since $n$ is odd, so for some $j$ value, $2^jr\equiv r \pmod n$ and so $|\hat{a}(r)|=1$. Thus either $\hat{a}(r)=0$ or $|\hat{a}(r)|=1$.
\end{proof}

\begin{theorem}
The set of idempotents of $\mathbb Z[\mathcal R_n]$ for $n$ odd is $\{0,e_x\;; x\in \mathcal R_n\}$.
\end{theorem}

\begin{proof}
Continuing with previous notations, let $m$ be the size of set $\{r \in \mathbb Z_n\;;\hat{a}(r)\neq 0\}$. Then using Lemma \ref{inv} and Lemma \ref{abval}, we get 
\[ m=n \sum\limits_{x\in \mathbb Z_n}a_x^2\]
Since $0\leq m\leq n$ and $\sum\limits_{x}a_x^2$ is a non-negative integer, so $m=0$ or $m=n$.
Since $u$ is an idempotent, under the augmentation map we have $\varepsilon(u)=0$ or $\varepsilon(u)=1$.

If $\varepsilon(u)=0$, then $ \hat{a}(0)=\sum\limits_{x}a_x=0$
Then $m = |\{\hat{a}(r)\neq 0\;; r\in \mathbb Z_n\}|\leq n-1$. So $m=0$. This implies $\hat{a}(r)=0$ for all $r\in \mathbb Z_n$. By using Lemma \ref{inv}, we get $a_x=0$ for all $x\in \mathcal R_n$ and so $u=0$.

If $\varepsilon(u)=1$, then $\hat{a}(0)=\sum\limits_{x}a_x=1$. Then $m = | \{\hat{a}(r)\neq 0\;; r\in \mathbb Z_n\}|\geq 1$. So $m=n$. This implies $\sum\limits_{x}a_x^2=1$. Since $a_x$ for all $x\in \mathcal R_n$ are integers, so the only solutions are $a_x=1$ for $x\in \mathcal{R}_n$ and $a_y=0$ for all $x\neq y\in \mathcal R_n$. Thus $u=e_x,\;x\in \mathcal R_n$.

\end{proof}
\section{Peirce Spectrum of $\mathbb{C}[\mathcal{R}_n]$} \label{sec3}
Let $\mathbb{C}[\mathcal{R}_n]$ be the complex quandle algebra of the dihedral quandle of odd order and $u \in \mathbb C[\mathcal {R} _ n]$ be an idempotent element. Define an mapping $S_u: \mathbb C[\mathcal R_n] \rightarrow \mathbb C[\mathcal R_n], w\mapsto wu$. The map $S_u$ is $\mathbb C$-linear. The Peirce spectrum of $\mathbb C[\mathcal R_n]$ is the set of eigenvalues of operators $S_u$ for all idempotent elements $u$ of $\mathbb C[\mathcal R_n]$. The above technique is also helpful in computing the Peirce spectrum of $\mathbb C[\mathcal R_n]$ for $n$ odd.

\begin{lemma}
Let $\mathbb C[\mathcal R_n]$ be the quandle ring of a dihedral quandle. Let $\zeta=e^{i2\pi/n}$. Then the set $\{f_r\;; r\in \mathbb Z_n\}$ where $f_r= \sum\limits_{x\in \mathbb Z_n} \zeta^{rx}e_x$ is a basis of $\mathbb C[\mathcal R_n]$ as $\mathbb C$-vector space.
\end{lemma}

\begin{proof}
We prove that the set $\{f_r\;;r\in \mathbb Z_n\}$ is a linear independent set. Let $\sum\limits_{r\in \mathbb Z_n} c_rf_r=0$. Then $\sum\limits_{r}c_r \left( \sum\limits_{x\in \mathbb Z_n} \zeta^{rx}e_x \right)=0$, and so $\sum\limits_{x\in \mathbb Z_n}\left( \sum\limits_{r\in \mathbb Z_n}c_r\zeta^{rx} \right)e_x=0$. Since $\{e_x\;; x\in \mathbb Z_n\}$ is a basis of $\mathbb C[\mathcal R_n]$, so $\sum\limits_{r\in \mathbb Z_n}c_r\zeta^{rx}=0$ for all $x\in \mathbb Z_n$. Now we fix $s\in \mathbb Z_n$, multiply the previous equality by $\zeta^{-sx}$ and sum for all $x\in \mathbb Z_n$. We get $\sum\limits_{x\in \mathbb Z_n} \left( \sum\limits_{r\in \mathbb Z_n} c_r\zeta^{rx}\right) \zeta^{-sx}= \sum\limits_{r\in \mathbb Z_n} c_r \sum\limits_{x\in \mathbb Z_n}\zeta^{(r-s)x}=0$. Now by Lemma \ref{ortho}, we get $nc_s=0$ and so $c_s=0$. Since $s$ is arbitrary, so $c_r=0$ for all $r\in \mathbb Z_n$. The size of set $\{f_r\;;r\in \mathbb Z_n\}$ is $n$, hence it spans $\mathbb C[\mathcal R_n]$.
\end{proof}

\begin{theorem}
The Peirce Spectrum of the complex quandle algebra of the dihedral quandle of odd order is $\{0,1,-1\}$.
\end{theorem}

\begin{proof}
Let $u=\sum\limits_{x}a_xe_x$ be an idempotent element of $\mathbb C[\mathcal R_n]$. Let $\zeta=e^{i2\pi/n}$. We define $\hat{a}(r)=\sum\limits_{x\in \mathbb Z_n}a_x\zeta^{rx}$.
Compute the operator $S_u: \mathbb C[\mathcal R_n] \rightarrow \mathbb C[\mathcal R_n], w \mapsto wu$ on the basis $\{f_r\;; r\in \mathbb Z_n\}$.
We get
\begin{align*}
S_u(f_r)=f_ru&=\left( \sum\limits_{x\in \mathbb Z_n}\zeta^{rx}e_x \right) \cdot \left( \sum\limits_{y\in \mathbb Z_n}a_ye_y\right)\\
&=\sum\limits_{x,y\in \mathbb Z_n}\zeta^{rx}a_ye_{2y-x}\\
&=\sum\limits_{y\in \mathbb Z_n}a_y\zeta^{2ry} \sum\limits_{z\in \mathbb Z_n}\zeta^{-rz}e_z\\
&=\hat{a}(2r)f_{-r}
\end{align*}
where the second last equality is obtained after substituting $z=2y-x$. After applying $S_u$ again we get $S_u^{2}(f_r)=\hat{a}(2r)S_u(f_{-r})=\hat{a}(2r)\hat{a}(-2r)f_r$. We denote $P(r)=\hat{a}(r)\hat{a}(-r)$. Then the preceding equation becomes $S^2_u(f_r)=P(2r)f_r$.
Thus $P(2r)$ an eigenvalue of $S^2_u$. 

Now from Lemma \ref{idemeq}, we have $\hat{a}(r)=\hat{a}(-r)\hat{a}(2r)$ for all $r\in \mathbb Z_n$. Replace $r$ by $-r$, then we have $\hat{a}(-r)=\hat{a}(r)\hat{a}(-2r)$. Multiplying these two equations gives $P(r)=P(r)P(2r)$ for all $r\in \mathbb Z_n$. One possibility is $P(r)=0$. If $P(r)\neq 0$, then $P(2r)=1$. Repeating this, we get $P(2^jr)=1$ for $j\geq 1$. Since $n$ is odd, so for some $j$, we have $2^jr\equiv r \pmod n$ and so $P(r)=1$. Thus, for $r\in \mathbb Z_n$ either $P(r)=0$ or $P(r)=1$.

Thus eigenvalues of $S_u^2$ belong to set $\{0,1\}$ and so the eigenvalues of $S_u$ belongs to $\{0,1,-1\}$. The eigenvalue $0$ is achieved by taking $u=0$, we have $S_u(w)=wu=0$ for all $w\in \mathbb{C}[\mathcal R_n]$. For $u=e_j$, and $w=e_j$, we get $S_u(w)=wu=e_je_j=e_j$. For $u=e_j$ and $w=e_x-e_{2j-x}$ for $x\neq j$. We get $S_u(w)=wu=(e_{x}-e_{2j-x})e_j=-(e_x-e_{2j-x})$.

\end{proof}

\section{Nontrivial idempotents in $\mathbb Z[{\rm Aff}(\mathbb Z_p,f)]$}\label{sec4}

The counterexample to the conjecture \ref{c1} is given in \cite[Theorem 5.1]{Jab26} in the integral quandle ring of affine latin quandle ${\rm Aff}(\mathbb Z_{37},f)$, for $f(x)=18x$. 

Each quandle automorphism linearly extends to the quandle ring automorphism, thus we have ${\rm Aut}(Q) \subseteq {\rm Aut}(\mathbb Z[Q])$. For an idempotent element $u\in \mathbb Z[Q]$ and for $f\in {\rm Aut}(\mathbb Z[Q])$, $f(u)$ is also an idempotent element since we have $f(u)^2=f(u^2)=f(u)$. As an example, if we apply the automorphism $e_x\mapsto e_{x+23}$ of $\mathbb Z[{\rm Aff}(\mathbb Z_{37},f)]$ for $f(x)=18x$, on the idempotent element of \cite[4,Theorem 5.1]{Jab26}, we get the following.

\begin{example}\label{eg1}
Let $Q={\rm Aff}(\mathbb Z_{37},f)$ be affine latin quandle on $\mathbb Z_{37}$ for $f(x)=18x$ then \begin{align*}
u=& e_0+2(e_3+e_4+e_7+e_{30}+e_{33}+e_{34})\\
&\hspace{0.35cm}+4(e_9+e_{12}+e_{16}+e_{21}+e_{25}+e_{28})\\
&\hspace{0.35cm}-2(e_1+e_5+e_6+e_8+e_{10}+e_{11}+e_{13}+e_{14}+e_{18}+e_{19}\\
&\hspace{0.995cm}+e_{23}+e_{24}+e_{26}+e_{27}+e_{29}+e_{31}+e_{32}+e_{36})
\end{align*}
is a nonzero idempotent in $\mathbb{Z}[Q]$.
\end{example}

For an affine latin quandle $Q={\rm Aff}(\mathbb Z_p,f)$ of prime order, we prove that if $\mathbb Z[Q]$ has a nontrivial idempotent $u$ of augmentation $1$, then with the help of such a nontrivial idempotent $u$, we can construct infinitely many nontrivial idempotents in $\mathbb Z[Q]$.

\begin{theorem}
Let the quandle ring $\mathbb Z[Q]$ for a connected quandle $Q={\rm Aff}(\mathbb Z_p,f)$ have a nontrivial idempotent $u$ with augmentation $1$. Then $\mathbb Z[Q]$ has infinitely many nontrivial idempotents.
\end{theorem}

\begin{proof}
Let $f(x)=tx$ for $t\in \mathbb Z_{p}^{*}$, then the quandle operation is $x\triangleright y=tx+(1-t)y$. Let $s=1-t$, then $s\neq 0$, since $Q$ is connected.  Let $u = \sum\limits_{x\in \mathbb Z_p}a_xe_x\in \mathbb Z[Q]$ be a nontrivial idempotent with augmentation $1$. Let $\zeta=e^{i2\pi/p}$ and define $\hat{a}(r)=\sum\limits_{x\in \mathbb Z_p}a_x\zeta^{rx}$ for all $r\in \mathbb Z_{p}$. We have $u^2=\sum\limits_{x,y\in \mathbb Z_p}a_xa_ye_{tx+sy}$ and so the coefficient of $e_z$ in $u^2$ is $\sum\limits_{tx+sy=z}a_xa_y$.

Construct the associative ring $R=\mathbb Z[X]/\langle X^{p}-1 \rangle$. Using the coefficients of the nontrivial idempotent $u$, define an element of $R$ as $A(X)=\sum\limits_{x=0}^{p-1}a_xX^x$. We have 
\[ A(X^{t})A(X^{s})= \sum\limits_{x,y\in \mathbb Z_p}a_xa_yX^{tx+sy}\]
The coefficients of $X^z$ modulo $X^{p}-1$ in the above expression
are exactly the coefficients of $e_z$ in $u^2$. Thus $u^2=u$ in $\mathbb Z[Q]$ is equivalent to $A(X^{t})A(X^{s})=A(X)$ in $R$.

Now for all positive integers $m\geq 1$, define $A_m(X)=A(X)^m$ and write $A_m(X)=\sum\limits_{z=0}^{p-1}a_z^{(m)}X^z \in R$. Using these coefficients, define $u_m=\sum\limits_{z \in \mathbb Z_p}a_z^{(m)}e_z$. Now we show that $u_m$ is an idempotent. It is enough to show $A_m(X^{t})A_{m}(X^{s})=A_m(X)$. We have
$A_m(X^{t})A_m(X^{s})=A(X^{t})^mA(X^{s})^m=(A(X^{t})A(X^{s}))^m=A(X)^m=A_m(X)$. Hence $u_m$ is an idempotent.

Now we show that $u_m\neq 0$. We have $\varepsilon(u)=\sum\limits_{x\in \mathbb Z_p}a_x=A(1)=1$. Now $\varepsilon(u_m)=\sum\limits_{x\in \mathbb Z_p}a_x^{(m)}=A_m(1)=A(1)^m=1^m=1$.

Now we show that for all $m\geq 1$, $u_m$ are distinct. Define a ring homomorphism $\phi: \mathbb Z[X]/\langle X^{p}-1 \rangle \rightarrow \mathbb C, F(X)\mapsto F(\zeta^r)$. We note that we can always choose $0\neq r\in \mathbb Z_{p}$ such that $|A(\zeta^r)|\neq 1$. Because if not, then $|A(\zeta^r)|=|\hat{a}(r)|=1$ for all $r\neq 0$, and since $A(\zeta^0)=A(1)=1$. Then Lemma \ref{inv} implies $\sum\limits_{x\in \mathbb Z_p}a_x^2=1$ and so $u$ is trivial, which is a contradiction. 
Moreover, we have $A(\zeta^r)\neq 0$, because if $A(\zeta^r)=0$, then $A(X)=c(1+X+\dots+X^{p-1})$ for some $c\in \mathbb Z$ and so $A(1)=cp\neq 1$.

Let $\rho=|A(\zeta^r)|$. Then $\rho>0$ and $\rho\neq 1$. Now $|A_m(\zeta^r)|=|A(\zeta^r)|^m=\rho^m$. Now for $m\neq k$, we get $\rho^m\neq \rho^k$ and so $A_m(\zeta^r)\neq A_k(\zeta^r)$ and consequently $A_m(X)\neq A_k(X)$. Thus $u_m\neq u_k$.

Finally, suppose $u_m=e_j$ for some $j\in \mathbb Z_p$. Then $A_m(X)=X^j$ and so $|A_m(\zeta^r)|=|\zeta^{rj}|=1$ but we have $|A_m(\zeta^r)|=|A(\zeta^r)|^m=\rho^m\neq 1$. Thus $u_m,m\geq 1$ are nontrivial, augmentation $1$, distinct idempotent elements of $\mathbb Z[Q]$.

\end{proof}

As an application of above theorem, for $u$ in Example \ref{eg1}, we have 
\begin{align*}
u_2=& 193e_0-56(e_1+e_{10}+e_{11}+e_{26}+e_{27}+e_{36})-68(e_2+e_{15}+e_{17}+e_{20}+e_{22}+e_{35})\\
&+4(e_3+e_4+e_7+e_{30}+e_{33}+e_{34})+44(e_{5}+e_{13}+e_{18}+e_{19}+e_{24}+e_{32})\\
&-24(e_{6}+e_{8}+e_{14}+e_{23}+e_{29}+e_{31})+68(e_{9}+e_{12}+e_{16}+e_{21}+e_{25}+e_{28})\\
\end{align*}

The image of $u_2$ under augmentation map $\varepsilon: \mathbb Z[Q] \rightarrow \mathbb Z, \sum a_ie_i\mapsto \sum a_i$ is $\sum\limits_{i=0}^{36} a_i=1$. The equality $u_2^2=u_2$ can be verified computationally.

\medskip \noindent \textbf{Acknowledgment.}
The authors acknowledge the use of AI-assisted tools for this article.

\printbibliography
\end{document}